\documentclass[11pt, a4paper]{article}
\usepackage{amsmath, amsthm, amssymb, url}
\usepackage[margin=1in]{geometry}
\usepackage[english]{babel}
\usepackage{multirow}
\usepackage{mathrsfs, enumerate}
\usepackage{authblk}

\newcommand{\End}{\mathrm{End}}

\newcommand{\supp}{\mathrm{supp}}

\newcommand{\id}{\mathrm{id}}

\newcommand{\CA}{\mathrm{CA}}

\newcommand{\EndCA}{\mathrm{EndCA}}

\newcommand{\Hom}{\mathrm{Hom}}

\theoremstyle{plain}

\newtheorem{corollary}{Corollary}
\newtheorem{lemma}{Lemma}

\newtheorem{theorem}{Theorem}

\theoremstyle{definition}
\newtheorem{definition}{Definition}

\newtheorem{example}{Example}
\newtheorem{remark}{Remark}

\begin{document}

\title{Cellular automata over algebraic structures}
\author{Alonso Castillo-Ramirez\footnote{alonso.castillor@academicos.udg.mx}, \ \  O. Mata-Guti\'errez\footnote{osbaldo.mata@academico.udg.mx}, and Angel Zaldivar-Corichi\footnote{angelus31415@gmail.com}  \\
\small{Department of Mathematics, University Centre of Exact Sciences and Engineering,\\ University of Guadalajara, Guadalajara, M\'exico.} }

\maketitle

\begin{abstract}
Let $G$ be a group and $A$ a set equipped with a collection of finitary operations. We study cellular automata $\tau : A^G \to A^G$ that preserve the operations of $A^G$ induced componentwise from the operations of $A$. We show that $\tau$ is an endomorphism of $A^G$ if and only if its local function is a homomorphism. When $A$ is entropic (i.e. all finitary operations are homomorphisms), we establish that the set $\EndCA(G;A)$, consisting of all such endomorphic cellular automata, is isomorphic to the direct limit of $\Hom(A^S, A)$, where $S$ runs among all finite subsets of $G$. In particular, when $A$ is an $R$-module, we show that $\EndCA(G;A)$ is isomorphic to the group algebra $\End(A)[G]$. Moreover, when $A$ is a finite Boolean algebra, we establish that the number of endomorphic cellular automata over $A^G$ admitting a memory set $S$ is precisely $(k \vert S \vert)^k$, where $k$ is the number of atoms of $A$.          \\

\end{abstract}

\section{Introduction}\label{intro}

The theory of cellular automata (CA) has important connections with many areas of mathema-tics, such as group theory, topology, symbolic dynamics, theoretical computer science, coding theory, and cryptography. In its classical setting, CA are studied over grids $\mathbb{Z}^d$ and finite alphabets $A$. However, in recent years, various generalisations have gained considerable interest: notably, CA over arbitrary groups instead of grids, \emph{linear CA} over vector spaces as alphabets, and \emph{additive CA} over commutative monoids as alphabets have been thoroughly investigated (e.g. see \cite{CSC06, CSC07,CSC10} and \cite[p. 952--953]{W02}). One-dimensional CA over arbitrary algebraic structures have been studied in \cite{ST12}. In this paper, we propose a general framework that encompasses all these settings by considering CA over arbitrary groups and alphabets with an arbitrary algebraic structure. 

For any group $G$ and any set $A$, let $A^G$ be the set of all functions of the form $x:G \to A$. A \emph{cellular automaton} over $A^G$ is a function $\tau : A^G \to A^G$ defined via a local function $\mu : A^S \to A$, where $S$ is a finite subset of $G$ called a \emph{memory set} for $\tau$. 

We shall assume that the set $A$ has an \emph{algebraic structure}, which means that $A$ is equipped with a collection of finitary operations. In this situation, $A^G$ inherits an algebraic structure via operations defined componentwise from the operations of $A$. Naturally, we restrict our attention to cellular automata that are \emph{endomorphisms} of $A^G$, i.e. cellular automata $\tau : A^G \to A^G$ that preserve the operations of $A^G$. Denote by $\CA(G;A)$ and $\EndCA(G;A)$ the sets of all CA over $A^G$ and all endomorphic CA over $A^G$, respectively. 

The paper is structured as follows. In Section \ref{sec:CA} we review the definition and basic facts of cellular automata over arbitrary groups. In Section \ref{sec:algebras}, we give a brief introduction to algebraic structures in the context of universal algebra, including their direct products and homomorphisms. In Section \ref{sec:main}, we show that the operations of $A^G$ are continuous in its prodiscrete topology and that the shift action of $G$ on $A^G$ preserves the operations. We establish that $\CA(G;A)$ is always an algebra of the same type as $A$, and, when $A$ is entropic (i.e. all its operations are homomorphisms), the set $\EndCA(G;A)$ is a subalgebra $\CA(G;A)$. Our main result of the section (Theorem \ref{th:local}) establishes that a cellular automaton is an endomorphism of $A^G$ if and only if its local function is a homomorphism. This implies, when $A$ is entropic, that $\EndCA(G;A)$ is isomorphic to the direct limit of the directed family $\{ \Hom(A^S, A) : S \subseteq G, \vert S \vert < \infty \}$. Finally, in Section \ref{sec:particular}, we focus on two particular situations: when $A$ is an $R$-module, we show that $\EndCA(G;A)$ is isomorphic to the group algebra $\End(A)[G]$ (Theorem \ref{th:group-algebra}), and when $A$ is a finite Boolean algebra, we establish that there are precisely  $(k \vert S \vert)^k$ endomorphic CA admitting a memory set $S \subseteq G$, where $k$ is the number of atoms of $A$ (Theorem \ref{th:boolean}). 

Our results, particularly Theorems \ref{th:local} and \ref{th:group-algebra}, generalise important facts known for linear CA. Moreover, we consider that the argument given in the proof of Theorem \ref{th:group-algebra} is more concise and elegant than the corresponding proof found in the literature for linear CA (c.f. \cite[Theorem 8.5.2.]{CSC10}); hence, we believe this illustrates the strength and beauty of introducing universal algebras in the context of CA.

\section{Cellular automata}\label{sec:CA}

Let $G$ be a group and $A$ a set. The set $A^G$ of all functions $x : G \to A$ is usually called the \emph{configuration space} in this context. The \emph{shift action} of $G$ on $A^G$ is defined by $g \cdot x(h) := x(g^{-1}h)$, for all $x \in A^G$, $g,h \in G$.

The following definition is taken from \cite[Sec.~1.4]{CSC10}.

\begin{definition}\label{defCA}
A \emph{cellular automaton} over $A^G$ is a function $\tau : A^G \to A^G$ such that there is a finite subset $S \subseteq G$, called a \emph{memory set} of $\tau$, and a \emph{local function} $\mu : A^S \to A$ satisfying
\[ \tau(x)(g) = \mu (( g^{-1} \cdot x) \vert_{S}),  \quad \forall x \in A^G, g \in G,  \]
where $\vert_{S}$ denotes the restriction to $S$ of a configuration in $A^G$. 
\end{definition}

\begin{remark}\label{re:local}
Let $\tau : A^G \to A^G$ be a cellular automaton with memory set $S$. The local defining function $\mu : A^S \to A$ may be recovered from $\tau : A^G \to A^G$ via
\[ \mu(y) = \tau(\overline{y}) (e), \quad \forall y \in A^S, \]
where $\overline{y} \in A^G$ is any extension of the function $y : S \to A$ and $e \in G$ is the identity element of the group. 
\end{remark}

\begin{remark}
A memory set for a cellular automaton $\tau : A^G \to A^G$ is normally not unique. Indeed, if $S \subseteq G$ is a memory set for $\tau$, with local defining function $\mu : A^S \to A$, then any superset $S^\prime \supseteq S$ is also a memory set for $\tau$: the local defining function $\mu^\prime : A^{S^{\prime}} \to A$ associated with $S^\prime$ is given by $\mu^\prime(x) = \mu(x\vert_S)$, for any $x \in A^{S^\prime}$.   
\end{remark}

The most famous example of a cellular automaton is John Conway's \emph{Game of Life}, which is defined over $\{0,1 \}^{\mathbb{Z}^2}$ and has memory set $S= \{ -1,0,1 \}^2$. 

A notorious family of examples are the so-called \emph{elementary cellular automata}, which are defined over $A^{\mathbb{Z}}$, with $A=\{0,1\}$, and have memory set $S= \{-1,0,1\}$; they are labeled as `Rule $M$', where $M$ is a number from $0$ to $255$. In each case, the local function $\mu_M : A^S \to A$ of Rule $M$ is determined as follows: let $M_1 \dots M_8$ be the binary representation of $M$ and write the elements of $A^S$ in lexicographical descending order, i.e. $111, 110, \dots, 000$; then, the image of the $i$-th element of $A^S$ under $\mu_M$ is $M_i$. 

\begin{example}\label{rule110}
Let  $G= \mathbb{Z}$ and $A = \{ 0,1 \}$. We may identify the elements of $A^\mathbb{Z}$ with bi-infinite sequences, i.e. for any $x \in A^\mathbb{Z}$, we may write
\[  x = ( \dots, x_{-2}, x_{-1}, \cdot x_0, x_1, x_2, \dots), \]
where $x_i = x(i) \in A$, for all $i \in \mathbb{Z}$, and the dot $\cdot$ is used to distinguish the image of zero in the sequence. Note that the action of $k \in \mathbb{Z}$ on $x \in A^\mathbb{Z}$ is given by 
\[ k \cdot x = ( \dots, x_{-k - 2}, x_{-k - 1}, \cdot x_{-k}, x_{-k + 1}, x_{-k +2}, \dots). \]
Let $S = \{ -1, 0 ,1 \} \subseteq G$ and define $\mu : A^S \to A$ by the following table 
\[ \begin{tabular}{c|cccccccc}
$x\in A^S$ & $111$ & $110$ & $101$ & $100$ & $011$ & $010$ & $001$ & $000$ \\ \hline
$\mu(x)$ & $0$ & $1$ & $1$ & $0$ & $1$ & $1$ & $1$ & $0$
\end{tabular}\] 
The cellular automaton $\tau: A^\mathbb{Z} \to A^\mathbb{Z}$ with memory set $S$ and local function $\mu$ as above is the elementary cellular automaton Rule 110. Remarkably, this cellular automaton is known to be Turing complete \cite{C04}.   
 \end{example}

\begin{example} \label{Ex2}
For any group $G$ and set $A$, fix a function $\phi : A \to A$. The map $\tau : A^G \to A^G$ defined by
\[ \tau(x) = \phi \circ x, \quad \forall x \in A^G \]
is cellular automaton over $A^G$ with memory set $S= \{e \}$ and local function $\mu : A^{S} \to A$ given by $\mu(y) = \phi(y(e))$, $\forall y \in A^S$. Indeed, we verify that, for any $x \in A^G$, $g \in G$,
\[ \tau(x)(g) = \phi(x(g)) = \phi( g^{-1} \cdot x (e) )= \mu( (g^{-1}\cdot x) \vert_{\{e \}}), \]
which satisfies Definition \ref{defCA}.
\end{example}

We endow $A^G$ with the \emph{prodiscrete topology}, which is the product topology of the discrete topology on $A$ (see \cite{M17} for a comprehensive introduction to product topologies). For every $g \in G$, the projection maps $\pi_g : A^G \to A$, defined by $\pi_g(x) := x(g)$, $\forall x \in A^G$, are continuous, and the preimage sets 
 \[  \pi_g^{-1}(a) = \{ x \in A^G : x(g) = a \}, \ \text{ for all } g \in G, a \in A, \]  
 form a subbasis of the prodiscrete topology of $A^G$ (i.e. every open set in $A^G$ may be written as a union of finite intersections of these preimage sets). A function $\tau : A^G \to A^G$ is continuous if and only if, for all $g \in G$, the functions $\pi_g \circ \tau : A^G \to A$ are continuous. 

It is known that every CA over $A^G$ commutes with the shift action (i.e. $\tau(g \cdot x) = g \cdot \tau(x)$, for all $x \in A^G$, $g \in G$) and is continuous in the prodiscrete topology of $A^G$. Moreover, when $A$ is finite, Curtis-Hedlund theorem \cite[Theorem 1.8.1]{CSC10} establishes that every shift commuting continuous function $\tau : A^G \to A^G$ is a cellular automaton. 

For any group $G$ and set $A$, define
\[ \CA(G;A) := \{ \tau : A^G \to A^G \; \vert \;  \tau \text{ is a cellular automaton} \}. \]
As the composition of any two cellular automata is a cellular automaton, the set $\CA(G;A)$ equipped with composition is a monoid (see \cite[Corollary 1.4.11]{CSC10}).   


\section{Algebraic structures} \label{sec:algebras}

In this section, we introduce some concepts and notation coming from universal algebra. For more details, see \cite{BS12}.

Let $A$ be a non-empty set and $n$ a non-negative integer. An \emph{$n$-ary operation} on $A$ is a function $f : A^n \to A$. A \emph{finitary operation} is an $n$-ary operation, for some $n$. When $n=0$, then $A^0 = \{ \emptyset \}$, so $0$-ary operation $f$ simply corresponds to a \emph{distinguished element} $e_f \in A$. 

An \emph{algebra type} is a set $\mathcal{F}$ of function symbols with a non-negative integer $n$ (the \emph{arity}) associated to each $f \in \mathcal{F}$: in such case, we say that $f \in \mathcal{F}$ is an $n$-ary function symbol. An \emph{algebraic structure of type $\mathcal{F}$}, or simply an \emph{algebra of type $\mathcal{F}$}, is a non-empty set $A$ together with a family $F$ of finitary operations on $A$ such that each $n$-ary operation $f^A \in F$ is indexed by an $n$-ary function symbol $f \in \mathcal{F}$.    

Let $A$ and $B$ algebras of the same type $\mathcal{F}$. We say that $A$ is a subalgebra of $B$ if $A$ is a subset of $B$, and for every $f \in \mathcal{F}$ we have $f^A = f^B \vert_A$. 

Given two algebras $A$ and $B$ of the same type $\mathcal{F}$, the \emph{direct product} $A \times B$ is an algebra of type $\mathcal{F}$ with componentwise operations: for every $n$-ary $f \in \mathcal{F}$, $a_1, \dots, a_n \in A$, $b_1, \dots, b_n \in B$,
\[ f^{A\times B}( (a_1, b_1), \dots, (a_n, b_n)) = (f^A(a_1, \dots, a_n), f^B(b_1, \dots, b_n)). \]
Hence, for each algebra $A$ and $m \geq 0$, we may define an algebra $A^m$. This definition of direct product may be generalised to arbitrary direct products.

A \emph{homomorphism} from $A$ to $B$ is a function $\phi : A \to B$ such that, for all $n$-ary $f \in \mathcal{F}$, $a_1, \dots, a_n \in A$, we have
\[ (\phi \circ f^A) (a_1, \dots a_n) = f^B ( \phi(a_1), \dots, \phi(a_n)).   \]
If $e_A$ and $e_B$ are distinguished elements of $A$ and $B$, respectively, corresponding to the same $0$-ary operation symbol, a homomorphism $\phi : A \to B$ satisfies $\phi(e_A) =e_B$. An \emph{endomorphism} of $A$ is simply a homomorphism from $A$ to $A$. Define the sets
\begin{align*}
\Hom(A,B) & := \{ \phi : A \to B \; \vert \; \phi \text{ is a homomorphism} \}, \\
\End(A) & := \Hom(A,A). 
\end{align*}

Two algebras $A$ and $B$ of type $\mathcal{F}$ are isomorphic if there exists a bijective homomorphism from $A$ to $B$; in such case, we write $A \cong B$

An algebra $A$ is called \emph{entropic} if for all $n$-ary $f \in \mathcal{F}$, we have $f^A \in \Hom(A^n, A)$ \cite{E62}. Entropic algebras are also known in the literature as \emph{medial}, \emph{commutative}, \emph{bi-commutative}, \emph{abelian}, among other names \cite{LP15}.

\begin{lemma}[\cite{K73}] \label{entropic}
Let $A$ be an algebra of type $\mathcal{F}$. The following statements are equivalent:
\begin{enumerate}
\item $A$ is entropic. 

\item For every $n$-ary $f \in \mathcal{F}$ and every $m$-ary $g \in \mathcal{F}$,
\[ f^A(g^A(a_{11}, \dots, a_{m1}),\dots, g^A(a_{1n},\dots, a_{mn})) = g^A ( f^A(a_{11}, \dots, a_{1n}),\dots f^A(a_{m1}, \dots, a_{mn})),  \]
for every $a_{ij} \in A$, $i=1,\dots, m$, $j = 1, \dots, n$.

\item For every algebra $X$ of type $\mathcal{F}$, $\Hom(X,A)$ is also an algebra of type $\mathcal{F}$ by defining 
\[ f^{\Hom(X,A)}(\phi_1, \dots, \phi_n)(x) := f^A ( \phi_1(x), \dots, \phi_n(x)), \]
for every $f \in \mathcal{F}$, $\phi_1, \dots, \phi_n \in \Hom(X,A)$, $x \in X$. 
\end{enumerate}
\end{lemma}

In particular, if $A$ is an entropic algebra, then $\End(A)$ is an algebra of the same type as $A$. Examples of entropic algebras are commutative semigroups (which include abelian groups) and modules over commutative rings (which include vector spaces). Note that a magma $A$ (i.e. an algebra with a single binary operation) is entropic if and only if $(a \cdot b) \cdot (c \cdot d) = (a \cdot c) \cdot (b \cdot d)$, for all $a, b, c, d  \in A$.

\section{Cellular automata over algebras} \label{sec:main}

Throughout this section, let $A$ be an algebra of type $\mathcal{F}$. Then, the configuration space $A^G$ is also an algebra of type $\mathcal{F}$: for any $n$-ary operation $f^A : A^n \to A$, we have an $n$-ary operation $f^{A^G} : (A^G)^n \to A^G$ induced componentwise as follows: for any $x_1, x_2, \dots, x_n \in A^G$ and $g \in G$, 
\[  f^{A^G}(x_1, x_2, \dots, x_n) (g) := f^A( x_1(g) , x_2(g), \dots, x_n(g)).  \]

A \emph{topological algebra of type $\mathcal{F}$} is an algebra $A$ of type $\mathcal{F}$ that is also a topological space in which all $n$-ary operations $f^A : A^n \to A$ are continuous (considering $A^n$ with the product topology). 

\begin{theorem}
With respect to the prodiscrete topology, $A^G$ is a topological algebra of type $\mathcal{F}$. 
\end{theorem}
\begin{proof}
Consider any $n$-ary $f \in \mathcal{F}$. First, observe that the operation $f^A : A^n \to A$ is continuous as both spaces $A^n$ and $A$ have the discrete topology (as the finite product of discrete spaces is discrete).

Fix $g \in G$ and consider the function $\pi^\prime_g : (A^G)^n \to A^n$ defined as $\pi^\prime_g(x_1, x_2, \dots, x_n) = (x_1(g), x_2(g), \dots, x_n(g))$, for every $x_1, x_2, \dots, x_n \in A^G$. The preimage of any $(a_1, a_2, \dots, a_n) \in A^n$ under $\pi^\prime_g$ is
\begin{align*}
(\pi^\prime_g)^{-1}(a_1, a_2, \dots, a_n) & = \{ (x_1, x_2, \dots, x_n) \in (A^G)^n : x_i(g) = a_i, \forall i  \}   \\
& = \{ x_1 \in A^G : x_1(g) = a_1 \} \times \dots \times \{ x_n \in A^G : x_n(g) = a_n \} \\
& = \pi^{-1}_g(a_1) \times \dots \times \pi^{-1}_g(a_n).
\end{align*}
This is an open set in $(A^G)^n $ as it is a Cartesian product of the open sets $\pi^{-1}_g(a_i)$ of $A^G$. Thus $\pi^\prime_g$ is a continuous function for any $g \in G$.

The operation $f^{A^G} : (A^G)^n \to A^G$ is continuous if and only if $\pi_g \circ f^{A^G} : (A^G)^n \to A$ is continuous for all $g \in G$. Note that
\[ \pi_g \circ f^{A^G} = f^A \circ \pi^\prime_g. \]
As both $f^A$ and $\pi^\prime_g$ are continuous, it follows that $\pi_g \circ f^{A^G}$ is continuous. 
\end{proof}


Let $G$ be a group acting on an algebra $X$. We say that $X$ is a \emph{$G$-algebra} if $G$ acts by homomorphisms, i.e. for all $g \in G$, the function $\varphi_g : X \to X$ defined by $\varphi_g(x) = g \cdot x$, $\forall x \in X$, is an endomorphism of the algebra. 

\begin{theorem}\label{th:action}
With respect to the shift action, $A^G$ is a $G$-algebra.   
\end{theorem}
\begin{proof}
Let $n \geq 1$ and observe that for all $n$-ary $ f \in \mathcal{F}$, $x_1, \dots, x_n \in A^G$, $g,h \in G$,
\begin{align*}
f^{A^G}( \varphi_g( x_1), \dots, \varphi_g( x_n))(h) & = f^A( g \cdot x_1(h), \dots, g \cdot x_n(h)) \\
& =   f^A( x_1(g^{-1}h), \dots, x_n(g^{-1}h)) \\
& = f^{A^G}( x_1, \dots, x_n) (g^{-1}h)  \\
& = g \cdot f^{A^G}( x_1, \dots, x_n) (h). 
\end{align*}
Thus, 
\[f^{A^G}( \varphi_g( x_1), \dots, \varphi_g( x_n))= (\varphi_g \circ  f^{A^G})( x_1, \dots, x_n), \]
which proves that $\varphi_g$ is an endomorphism of $A^G$. 

When $n=0$, a distinguished element $e_{A^G} \in A^G$ induced componentwise from a distinguished element $e_{A} \in A$ is defined by $e_{A^G}(g) = e_{A} \in A$, for all $g \in G$. Hence, for all $g,h \in G$, we have
\[ g \cdot e_{A^G} (h) = e_{A^G} (g^{-1}h) = e_A = e_{A^G}(h). \]
Thus, $\varphi_g(e_{A^G}) = e_{A^G}$ for all $g \in G$. 
\end{proof}

\begin{lemma}\label{subalgebra}
The set $\CA(G;A)$ is a subalgebra of $(A^G)^{A^G}$.
\end{lemma}
\begin{proof}
Let $X :=(A^G)^{A^G}$. We must show that for any $n$-ary $f \in \mathcal{F}$ and $\tau_1, \dots, \tau_n \in \CA(G;A)$ we have $f^{X}(\tau_1, \dots, \tau_n) \in \CA(G;A)$. Let $S_i$ and $\mu_i : A^{S_i} \to A$ be the memory set and local function of $\tau_i$, respectively, for $i=1,\dots, n$. Define $S:=\bigcup_{i=1}^n S_i$ and $\mu : A^S \to A$ by 
\[ \mu(y) := f^A( \mu_1( y \vert_{S_1} ), \dots, \mu_n( y \vert_{S_n} ) ), \ \ \ \forall y \in A^S.   \]
Then, for all $x\in A^G, g \in G$,
\begin{align*}
f^{X}(\tau_1, \dots, \tau_n)(x)(g) & = f^{A^G} (\tau_1(x), \dots, \tau_n(x))(g) \\
& = f^A( \tau_1(x)(g), \dots, \tau_n(x)(g)) \\
& = f^A( \mu_1( (g^{-1} \cdot x) \vert_{S_1} ), \dots, \mu_n( (g^{-1} \cdot x) \vert_{S_n} ) )\\ 
& = \mu( (g^{-1} \cdot x) \vert_{S} ). 
 \end{align*}
This shows that $f^{X}(\tau_1, \dots, \tau_n)$ is a cellular automaton with memory set $S$ and local function $\mu : A^S \to A$. 
\end{proof}

For any $S \subset G$, define
\[ \CA(G,S; A) := \{ \tau : A^G \to A^G \; \vert \; \tau \text{ is a cellular automaton with memory set contained in $S$} \}. \]
This is not a monoid under composition, as the memory set of $\tau \circ \sigma$, for $\tau, \sigma \in \CA(G,S; A)$, may not be contained in $S$. However, a similar argument as in Lemma \ref{subalgebra} shows that $\CA(G,S; A)$ is a subalgebra of $\CA(G;A)$.  

Naturally, we are interested on cellular automata that preserve the algebraic structure of $A^G$. Define 
\begin{align*}
\EndCA(G,S;A) &:=  \End(A^G) \cap \CA(G,S;A), \\
\EndCA(G;A) &:= \End(A^G) \cap \CA(G:A). 
\end{align*}
We call the elements of $\EndCA(G;A)$ \emph{endomorphic cellular automata} over $A^G$. 

\begin{lemma}
If $A$ is an entropic algebra, then both $\EndCA(G,S;A)$ and $\EndCA(G;A)$ are subalgebras of $\CA(G:A)$ 
\end{lemma}
\begin{proof}
By Lemma \ref{entropic}, $\End(A^G)$ is an algebra with the induced componentwise operations. Hence, being intersections of algebras, both $\EndCA(G,S;A)$ and $\EndCA(G;A)$ are algebras.
\end{proof}

\begin{example}\label{ex-sum}
Suppose that $A$ is an abelian group with operation $+$, and consider a finite subset $S \subseteq G$. Define $\mu : A^S \to A$ by
\[ \mu(z) = \sum_{s \in S} z(s), \quad \forall z \in A^S.   \] 

The cellular automaton $\tau : A^G \to A^G$ defined by the local function $\mu :A^S \to A$ preserves the addition of $A^G$. Indeed, for all $x,y \in A^G$, $g \in G$,
\begin{align*}
 \tau(x + y)(g) & = \mu( (g^{-1} \cdot (x + y) \vert_{S} )  \\[.7em]
& = \sum_{s \in S} g^{-1} \cdot (x+y) (s)  \\[.7em]
& = \sum_{s \in S} g^{-1} \cdot x (s) + \sum_{s \in S} g^{-1} \cdot y (s)  \\[.7em]
& = \tau(x)(g) + \tau(y)(g).
\end{align*}
Therefore, $\tau(x+y) = \tau(x) + \tau(y)$, for all $x,y \in A^G$, so $\tau \in \EndCA(G;A)$. 
\end{example}

\begin{example}
If we consider $A=\{0,1\}$ as an abelian group with addition modulo 2, then the elementary cellular automaton $\tau : A^{\mathbb{Z}} \to A^{\mathbb{Z}}$ known as Rule 110 (defined in Example \ref{rule110}) does not preserve the addition of $A^{\mathbb{Z}}$. This may be verified, for example, by defining $x, y \in A^{\mathbb{Z}}$ as 
\[ x = (\dots, 0,0, \cdot 1, 0,0, \dots)  \quad \text{and} \quad y = (\dots, 0, 0, \cdot 0 , 1, 0 ,\dots), \]
and noting that $\tau(x) + \tau(y) \neq \tau(x+y)$. 
\end{example}

\begin{example}
Let $A$ be any algebra and fix $\phi \in \End(A)$. We claim that the map $\tau : A^G \to A^G$ defined by
\[ \tau(x) = \phi \circ x, \quad \forall x \in A^G \]
is an endomorphic cellular automaton over $A^G$. By Example \ref{Ex2}, this is a cellular automaton with memory set $S = \{ e \}$ and local function $\mu : A^{S} \to A$ given by $\mu(y) = \phi(y(e))$, $\forall y \in A^S$. Now we check that $\tau$ is an endomorphism of $A^G$: for every $n$-ary $f \in \mathcal{F}$, $x_1, x_2, \dots, x_n \in A^G$, and $g \in G$, we have
\begin{align*}
f^{A^G}( \tau(x_1) , \tau(x_2), \dots, \tau(x_n)) (g) & =f^{A^G}( \phi \circ x_1  , \phi \circ x_2, \dots,  \phi \circ x_n)(g)    \ \\
 & =  f^A( \phi \circ x_1(g) , \phi \circ x_2(g), \dots,  \phi \circ x_n(g)) \\
 & = \phi \circ f^A( x_1(g) , x_2(g), \dots,  x_n(g)) \\
 & = (\tau \circ f^{A^G})(x_1, x_2, \dots, x_n) (g).
 \end{align*}
This generalises Example 8.1.3. (b) in \cite{CSC10} which was given for linear cellular automata. 
\end{example}

\begin{example}
Let $A$ be any algebra and fix $g_0 \in G$. Consider the bijection $R_{g_0}$ of $G$ defined by the right multiplication by $g_0$, i.e. $R_{g_0} : G \to G$ is the map defined by $R_{g_0}(g) = gg_0$, for all $g \in G$. We claim that the map $\tau : A^G \to A^G$ defined by
\[ \tau(x) =  x \circ R_{g_0}, \quad \forall x \in A^G \]
is an endomorphic cellular automaton over $A^G$. Note that $\tau$ is a cellular automaton with memory set $S = \{ g_0 \}$ and local function $\mu : A^{S} \to A$ given by $\mu(y) = y(g_0)$, $\forall y \in A^S$. We check that $\tau$ is an endomorphism of $A^G$: for every $n$-ary $f \in \mathcal{F}$, $x_1, x_2, \dots, x_n \in A^G$, and $g \in G$, 
\begin{align*}
f^{A^G}( \tau(x_1) , \tau(x_2), \dots, \tau(x_n)) (g) & = f^A( x_1\circ R_{g_0}(g) , x_2\circ R_{g_0}(g), \dots,  x_n\circ R_{g_0}(g)) \\
 & = f^A( x_1(gg_0) , x_2(gg_0), \dots,  x_n(gg_0)) \\
 & = f^{A^G}(x_1, x_2, \dots, x_n)(gg_0) \\ 
 & = f^{A^G}(x_1, x_2, \dots, x_n) \circ R_{g_0} (g) \\
 & = (\tau \circ f^{A^G})(x_1, x_2, \dots, x_n)(g) 
 \end{align*}
This generalises Example 8.1.3. (c) in \cite{CSC10} which was given for linear cellular automata. 
\end{example}

The following two lemmas are technical, but they are required to show the main result of this section. 

\begin{lemma}\label{le:restriction}
Consider an $n$-ary operation $f^A : A^n \to A$ and $S \subseteq G$. For any $x_1, \dots, x_n \in A^G$, we have
\[ f^{A^G} (x_1, \dots, x_n) \vert_S = f^{A^S} (x_1\vert_S, \dots, x_n \vert_S). \]
\end{lemma}
\begin{proof}
The above equality compares two elements of $A^S$. Note that, for every $s \in S$, the left-hand-side gives us 
\[f^{A^G} (x_1, \dots, x_n) \vert_S(s) = f^A(x_1(s), \dots, x_n(s)). \]
On the other hand, for every $s \in S$, the right-hand-side gives us
\[  f^{A^S} (x_1\vert_S, \dots, x_n \vert_S)(s) = f^A(x_1(s), \dots, x_n(s)). \]
The result follows. 
\end{proof}

For any $S \subseteq G$ and $y \in A^S$, denote by $\overline{y}$ an element of $A^G$ such that $\overline{y} \vert_S = y$.  

\begin{lemma}\label{bar}
Let $S \subseteq G$. For any $n$-ary $f \in \mathcal{F}$ and $y_1, \dots, y_n \in A^S$, we have
\[   \overline{f^{A^S}(y_1, \dots, y_n)} \vert_S = f^{A^G}(\overline{y_1}, \dots, \overline{y_n} ) \vert_S. \]
\end{lemma}
\begin{proof}
For every $s \in S$ we have
\begin{align*}
\overline{f^{A^S}(y_1, \dots, y_n)}(s) & =  f^{A^S}(y_1, \dots, y_n) (s) \\
&  f^{A}(y_1(s), \dots, y_n(s) ) \\
& f^{A}( \overline{y_1}(s), \dots, \overline{y_n} (s) )  \\
& f^{A^G} ( \overline{y_1}, \dots, \overline{y_n})(s).  
\end{align*}
\end{proof}

The following result shows that a cellular automaton is an endomorphism of $A^G$ if and only if its local function is a homomorphism. This is a significant generalisation of \cite[Proposition 8.1.1.]{CSC10}, which was proved for linear cellular automata. 

\begin{theorem}\label{th:local}
Let $\tau : A^G \to A^G$ be a cellular automaton with memory set $S \subseteq G$ and local function $\mu :A^S \to A$. Then $\tau \in \End(A^G)$ if and only if $\mu \in \Hom(A^S, A)$. 
\end{theorem}
\begin{proof}
Suppose that $\tau \in \End(A^G)$. Recall (see Remark \ref{re:local}) that $\mu(y) = \tau(\overline{y})(e)$ for every $y \in A^S$. Then, for any $n$-ary $f \in \mathcal{F}$ and $y_1, \dots y_n \in A^S$ we have
\begin{equation} \label{eq1}
  \mu( f^{A^S}(y_1, \dots, y_n)) = \tau( \overline{f^{A^S}(y_1, \dots, y_n)}) (e). 
  \end{equation}
It follows from the definition of cellular automata that if $x_1, x_2 \in A^G$ are two configurations such that $x_1 \vert_S = x_2 \vert_S$, then $\tau(x_1)(e) = \tau(x_2)(e)$. Hence, by Lemma \ref{bar} we have
\begin{equation}\label{eq2}
\tau( \overline{f^{A^S}(y_1, \dots y_n)}) (e) = \tau( f^{A^G}(\overline{y_1}, \dots, \overline{y_n} ) ) (e). 
\end{equation}
Combining (\ref{eq1}) and (\ref{eq2}), and using the fact that $\tau \in \End(A^G)$, we obtain 
\begin{align*}
 \mu( f^{A^S}(y_1, \dots y_n)) & = \tau( f^{A^G}(\overline{y_1}, \dots, \overline{y_n} ) ) (e) \\
 & = f^{A^G}( \tau(\overline{y_1}), \dots, \tau(\overline{y_n}) ) (e) \\
 & = f^A ( \tau(\overline{y_1})(e), \dots, \tau(\overline{y_n})(e) ) \\
 & = f^A( \mu(y_1), \dots, \mu(y_n)).
\end{align*}
Therefore, $\mu \in \Hom(A^S, A)$. 

Suppose now that $\mu \in\Hom(A^S, A)$. For any $n$-ary $f \in \mathcal{F}$, $x_1, \dots, x_n \in A^G$, and $g \in G$, we apply the definition of cellular automaton 
\[ \tau \left( f^{A^G} (x_1, \dots, x_n) \right) (g)  = \mu( (g^{-1} \cdot   f^{A^G} (x_1, \dots, x_n)) \vert_S ). \]

By Theorem \ref{th:action} and Lemma \ref{le:restriction},
\begin{align*}
\mu( (g^{-1} \cdot   f^{A^G} (x_1, \dots, x_n)) \vert_S ) & =  \mu( (f^{A^G} (g^{-1} \cdot x_1, \dots, g^{-1} \cdot x_n)) \vert_S ) \\
& = \mu ( f^{A^S} ( (g^{-1} \cdot x_1)\vert_S , \dots, (g^{-1} \cdot x_n)\vert_S )) \\
& = f^A ( \mu((g^{-1} \cdot x_1)\vert_S ), \dots, \mu((g^{-1} \cdot x_n)\vert_S) ) \\
& = f^A ( \tau(x_1)(g), \dots, \tau(x_n)(g)) \\
& = f^{A^G}( \tau(x_1), \dots, \tau(x_n))(g).
\end{align*}

Hence $\tau  \in \End(A^G)$. 
\end{proof}

\begin{corollary}\label{cor-hom}
The number of endomorphic cellular automata admitting a memory set $S \subseteq G$ is $\vert \Hom(A^S , A) \vert$. 
\end{corollary}

\begin{theorem}\label{th:CA-S}
Let $A$ be an entropic algebra. Then,
\[  \EndCA(G,S;A) \cong \Hom(A^S, A). \]
\end{theorem}
\begin{proof}
Consider the function $\Phi : \EndCA(G,S;A) \to  \Hom(A^S, A)$ defined by $\Phi(\tau) = \mu$, where $\tau \in  \EndCA(G,S;A)$ and $\mu : A^S \to A$ is the local function of $\tau$. By Theorem \ref{th:local}, $\Phi$ is well-defined, and it is easy to see that $\Phi$ is bijective (using Remark \ref{re:local}). In order to show that $\Phi$ is a homomorphism, let $f \in \mathcal{F}$ be an $n$-ary function symbol. Let $\tau_1, \dots, \tau_n \in \EndCA(G,S;A)$ have local functions $\mu_1, \dots , \mu_n \in \Hom(A^S, A)$, respectively. We claim that the local function of $f^{\EndCA(G,S;A)}(\tau_1, \dots, \tau_n)$ is $f^{\Hom(A^S, A)}(\mu_1, \dots, \mu_n)$. Indeed, for all $x \in A^G$, we have 
\begin{align*}
f^{\EndCA(G,S;A)}(\tau_1, \dots, \tau_n)(x)(e) & =  f^{A}(\tau_1(x)(e) , \dots, \tau_n(x)(e) ) \\
& = f^{A}(\mu_1( x \vert_S) , \dots, \mu_n( x \vert_S) ) \\
&  = f^{\Hom(A^S, A)}(\mu_1, \dots, \mu_n)(x\vert_S).
\end{align*}
This shows that 
\[ \Phi (f^{\EndCA(G,S;A)}(\tau_1, \dots, \tau_n)) = f^{\Hom(A^S, A)}( \Phi(\tau_1), \dots, \Phi(\tau_n)),   \]
proving that $\Phi$ is an isomorphism.
\end{proof}

A partially ordered set $(I, \leq)$ is a \emph{directed set} if for every pair $i,j \in I$ there exists $z \in I$ such that $i \leq z$ and $j \leq z$. A \emph{directed family of algebras} is collection $\mathcal{A}:=\{ A_i : i \in I \}$ of algebras of the same type $\mathcal{F}$ indexed by a directed set $I$ together with a collection of maps $\{ \phi_{ij} : A_i \to A_j : i, j \in I, i \leq j \}$, such that $\phi_{ij} \circ \phi_{jk} = \phi_{ik}$ if $i\leq j \leq k$ and $\phi_{ii}$ is the identity map for all $i \in I$. In this situation, as shown in \cite[Section 21]{G08}, one may construct an algebra of type $\mathcal{F}$ called the \emph{direct limit} of $\mathcal{A}$ and denoted by $\lim \limits_{\to}(\mathcal{A})$.    

Let $\mathcal{P}_{<\infty}(G)$ be the set of all finite subsets of $G$. Ordered by inclusion, $\mathcal{P}_{<\infty}(G)$ is a directed set as an upper bound for a pair $S_1, S_2 \in   \mathcal{P}_{<\infty}(G)$ is $S_1 \cup S_2 \in \mathcal{P}_{<\infty}(G)$. Together with the inclusion maps, the collection $\mathcal{A}:=\{ \EndCA(G, S;A) : S \in \mathcal{P}_{<\infty}(G) \}$ is a directed family and its direct limit is given by its union:
\[ \lim_{\to}(\mathcal{A}) =  \bigcup_{S \in \mathcal{P}_{<\infty}(G)} \EndCA(G,S;A).  \]

As every cellular automaton has a finite memory set by definition, it is easy to see that
\[ \EndCA(G;A) = \bigcup_{S \in \mathcal{P}_{<\infty}(G)} \EndCA(G,S;A).  \]

Hence, Theorem \ref{th:CA-S} implies the following result. 

\begin{corollary}
Let $A$ be an entropic algebra. The algebra $\EndCA(G;A)$ is isomorphic to the direct limit of the directed family $\{ \Hom(A^S, A) : S \in \mathcal{P}_{<\infty}(G)  \}$.  
\end{corollary}

\section{Cellular automata over particular algebras} \label{sec:particular}

\subsection{$R$-modules}

In this section, we assume that $R$ is a ring (associative with $1$) and $M$ is an $R$-module. In the language of universal algebra, this means that $M$ is an algebra of type $\mathcal{F} = \{ +, 0, r \cdot \; \vert \; r \in R \}$, where $+$ is binary operation, $0$ is a distinguished element, and $\{ r \cdot \; \vert \; r \in R\}$, is a family of unary operations. Important classes of $R$-modules are vector spaces (in which $R$ is a field) and abelian groups (in which $R=\mathbb{Z}$).


For any $x \in M^G$, define the support of $x$ as the set $\supp(x) := \{ g \in G : x(g) \neq  0 \}$. Let $M[G]$ be the subset of $M^G$ consisting of all configurations with finite support; this is an $R$-module with operations induced componentwise. Furthermore, we define a new binary operation on $M[G]$, called the \emph{convolution product}, as follows 
\[ (x \cdot y)(g) := \sum_{h \in G} x(h) y (h^{-1}g), \quad \forall x, y \in M[G].  \]
Notice that the sum above is finite because both $x$ and $y$ have finite support. In the case when $M=R$, $R[G]$ is known as the \emph{group ring} of $G$ over $R$, and has significant importance in the representation theory of $G$ (see \cite{MS02}).    

Equivalently, $M[G]$ may be seen as the set of formal sums $\sum_{g \in G} m_g g$, such that $m_g \in M$, for all $g \in G$, and the set $\{ g \in G : m_g \neq 0 \}$ is finite. 

\begin{remark}
As $M[G]$ consists of functions with finite support, we have
\[ M[G]  = \bigcup_{S \in \mathcal{P}_{<\infty}(G)} M^S.  \] 
\end{remark}  

If $A$ is an $R$-module, then both $\End(A)$ and $\End(A)[G]$ are $Z(R)$-modules, where $Z(R)$ denotes the center of $R$. The following result is a generalisation of \cite[Theorem 8.5.2.]{CSC10}.

\begin{theorem}\label{th:group-algebra}
Let $A$ be an $R$-module and $G$ a group. Then, $\EndCA(G;A) \cong \End(A)[G]$ as $Z(R)$-modules. 
\end{theorem}
\begin{proof}
If $A$ is an $R$-module, then $\Hom(A^2, A) \cong \End(A) \times \End(A)$ (see \cite[Corollary 2.32]{R09}). Using this and Theorem \ref{th:CA-S}, we obtain that for every finite $S \subseteq G$,
\[ \EndCA(G,S;A) \cong \Hom(A^S, A) \cong \prod_{s \in S} \End(A) = \End(A)^S.  \]
The result follows because
\[ \EndCA(G;A) \cong \bigcup_{S \in \mathcal{P}_{<\infty}(G)} \Hom(A^S, A) \cong \bigcup_{S \in \mathcal{P}_{<\infty}(G)} \End(A)^S = \End(A)[G]. \]
\end{proof}

An explicit isomorphism from $\Hom(A^S, A)$ to $\prod_{s \in S} \End(A)$ is defined as follows: the image of $\mu : A^S \to A$ is $\prod_{s \in S} (\mu \circ j_s)$, where $j_s : A \to A^S$ is the natural embedding given by
\[ j_s(a)(r) := \begin{cases}
a & \text{if } r=s \\
0 & \text{otherwise,}
\end{cases} \quad \forall a \in A, r \in S. \]
Hence, an explicit isomorphism $\Psi :  \EndCA(G;A) \to \End(A)[G]$ may be defined as follows: the image of a cellular automaton $\tau : A^G \to A^G$ with memory set $S$ and local function $\mu : A^S \to A$ corresponds to the element $\sum_{s \in S} (\mu \circ j_s) s$ in $\End(A)[G]$, given in the formal sums notation. From here we see that the minimal memory set of $\tau$ corresponds to the support of $\Psi(\tau)$. It follows, by the same proof as in \cite[Theorem 8.5.2.]{CSC10} that $\Psi$ satisfies
\[ \Psi( \tau \circ \sigma) = \Psi(\tau) \Psi(\sigma), \quad \forall \tau, \sigma \in \EndCA(G;A), \] 
where the operation on the right-hand side is the convolution product. 

\begin{example}
Let $A$ be an abelian group. Using the formal sums notation, the cellular automaton $\tau : A^G \to A^G$ with memory set $S \subseteq G$ and local function $\mu : A^S \to A$ defined in Example \ref{ex-sum} corresponds to the element $\sum_{s \in S} \id \; s$ of $\End(A)[G]$, where $\id : A \to A$ is the identity endomorphism.     
\end{example}

\begin{example}
Let $A$ be an $R$-module. Using the formal sums notation, the cellular automaton $\tau : A^G \to A^G$ defined via $\phi \in \End(A)$ as in Example \ref{Ex2} corresponds to the element $\phi e$ of $\End(A)[G]$.
\end{example}

\begin{remark}
By Theorem \ref{th:group-algebra}, the number of cellular automata in $\EndCA(G;A)$ admitting a memory set $S \subseteq G$ is $\vert \End(A) \vert^{\vert S \vert}$. This formula is specially useful when $\End(A)$ is finite.   
\end{remark}

\begin{corollary}
Let $G$ be a group and $S \subseteq G$ a finite subset of size $s$.
\begin{enumerate}
\item Suppose $A$ is a vector space of dimension $n < \infty$ over a finite field $F$. The number of linear cellular automata $\tau : A^G \to A^G$ admitting a memory set $S$ is $\vert F \vert^{n^2 s }$.

\item Let $A$ be a group isomorphic to $\mathbb{Z}_n$. The number of cellular automata $\tau : A^G \to A^G$ that are group homomorphisms, or $\mathbb{Z}$-module homomorphisms, admitting a memory set $S$ is $n^{s}$. 
\end{enumerate}
\end{corollary}
\begin{proof}
When $A$ is a vector space, it is well-known that $\End(A)$ is isomorphic to the algebra of matrices $M_{n \times n}(F)$. Hence, $\vert \End(A) \vert = \vert M_{n \times n}(F) \vert = \vert F \vert^{n^2}$. Part (1.) follows by the previous remark.

For part (2.) it is easy to see that $\vert \End(\mathbb{Z}_n) \vert = n$, so the result again follows by the previous remark. 
\end{proof}

\begin{example}
Consider $A =\{ 0,1\}$ as a group isomorphic to $\mathbb{Z}_2$. Then, the number of \emph{linear} (or \emph{additive}) elementary cellular automata over $A^{\mathbb{Z}}$ is $2^3= 8$, as elementary CA admit a memory set $S=\{-1,0,1 \}$; explicitly, these are Rules 0, 60, 90, 102, 150, 170, 204 and 240.    
\end{example}


\subsection{Boolean algebras}

A \emph{Boolean algebra} is a set $B$ equipped with two binary operations $\land : B \times B \to B$ and $\lor : B \times B \to B$ called \emph{meet} and \emph{join}, respectively, an unary operation $\neg : B \to B$ called \emph{complement} and two distinguished elements $0 \in B$ and $1 \in B$, called bottom and top, respectively. The meet and join are commutative and associative, and distributive between each other, $0$ is an identity for $\lor$, $1$ is an identity of $\land$, and complements satisfy $x \lor \neg x = 1$ and $x \land \neg x = 0$, for all $x \in B$ (see \cite{GH09} for details). 

The smallest example of a Boolean algebra is $\textbf{2} := \{ 0,1 \}$, which only contains the two distinguished elements. If $X$ is a set, the power set $\mathcal{P}(X)$ with meet, join and complement given by the intersection, union and complement of subsets, respectively, is a Boolean algebra. In fact, $\mathcal{P}(X) \cong \textbf{2}^{X}$ and every finite Boolean algebra is isomorphic to $\mathcal{P}(S)$, for some finite non-empty set $S$ (\cite[Ch. 15]{GH09}).

%

We may define a partial order relation $\leq$ on a Boolean algebra $B$ as follows: $x \leq y$ if and only if $x \lor y = y$ (or, equivalently, if and only if $x \land y = x$). The non-zero minimal elements of $B$ with respect to this order are called the \emph{atoms} of $B$. For example, when $B$ is the power set $\mathcal{P}(S)$ of a set $S$, the atoms are precisely the singleton sets $\{ x\}$ with $x \in S$. Thus, $\mathcal{P}(S)$ has exactly $\vert S \vert$ atoms.   

An \emph{ideal} of a Boolean algebra $B$ is a subset $I \subseteq B$ such that $0 \in I$, for all $a,b \in I$, we have $a \lor b \in I$, and, for all $a \in I$, $x \in B$, we have $a \land x \in I$. An ideal $I$ of $B$ is \emph{maximal} if $I$ is properly contained in $B$ and there is no ideal $J$ of $B$ such that $I \subset J \subset B$. The \emph{ideal generated} by a subset $E$ of $B$, denoted by $\langle E \rangle$, is the smallest ideal of $B$ that contains $E$. An ideal generated by a singleton is called a \emph{principal ideal}. In fact, for any $y \in B$, a theorem of Stone (see \cite[Ch. 18]{GH09}) implies that
\[ \langle y \rangle = \{ x \in B : x \leq y  \}. \] 

The kernel of a Boolean homomorphism $\phi : B \to B^\prime$ between Boolean algebras is the set 
\[ \ker(\phi) := \{ b \in B : \phi(b) = 0 \}, \]
Kernels of Boolean homomorphisms are always ideals.

\begin{lemma}\label{le:boolean}
Let $G$ be a group and consider the Boolean algebra $\textbf{2}= \{0,1\}$. Then, the number of Boolean cellular automata over $\textbf{2}^G$ with memory set $S \subseteq G$ is $\vert S \vert$. 
\end{lemma}
\begin{proof}
By Corollary \ref{cor-hom}, the number of Boolean cellular automata over $\textbf{2}^G$ admitting memory set $S \subseteq G$ is equal to $\vert \Hom(\textbf{2}^S, \textbf{2}) \vert$. By \cite[Ch. 22]{GH09}, the set  $\Hom(\textbf{2}^S, \textbf{2})$ is in bijection with the set of maximal ideals of $\textbf{2}^S$ via $\mu \mapsto \ker(\mu)$, for every $\mu \in \Hom(\textbf{2}^S, \textbf{2})$. As $\textbf{2}^S$ is finite, every ideal is principal \cite[Corollary 18.2]{GH09}, so the maximal ideals of $\textbf{2}^S$ are given by $\langle \neg a \rangle$, where $a \in \textbf{2}^S$ is an atom. As $\textbf{2}^S \cong \mathcal{P}(S)$ has precisely $\vert S \vert$ atoms (corresponding to the singletons $\{ s\}$, $s \in S$), then $\textbf{2}^S$ has precisely $\vert S \vert$ maximal ideals, and the result follows. 
\end{proof}

\begin{corollary}
Let $G$ be a group and consider the Boolean algebra $\textbf{2}= \{0,1\}$. Let $\tau : \textbf{2}^G  \to \textbf{2}^G$ be a cellular automaton with memory set $S$. Then, $\tau$ is a Boolean homomorphism if and only its local function is a projection $\pi_s : \textbf{2}^S \to \textbf{2}$, for some $s \in S$. 
\end{corollary}
\begin{proof}
Observe that $\pi_s \in \Hom(\textbf{2}^S, \textbf{2})$, for each $s \in S$, and use Lemma \ref{le:boolean} and Theorem \ref{th:local}.
\end{proof}

\begin{example}
There are precisely $3$ elementary cellular automata over $\textbf{2}^\mathbb{Z}$ that are Boolean homomorphisms; explicitly, these are Rules 170, 204, and 240.
\end{example}

\begin{theorem}\label{th:boolean}
Let $G$ be a group and $A$ be a finite Boolean algebra. The number of Boolean cellular automata over $A^G$ with memory set $S \subseteq G$ is $(k \vert S \vert)^k$, where $k$ is the number of atoms of $A$. 	
\end{theorem}
\begin{proof}
By Corollary \ref{cor-hom}, we must show that $\vert \Hom(A^S, A) \vert = (k \vert S \vert)^k$. By \cite[Corollary 15.1]{GH09}, $A$ is isomorphic to $\textbf{2}^k$, where $k$ is the number of atoms of $A$, and $A^S$ is isomorphic to $\textbf{2}^{ks}$, where $s=\vert S \vert$, so 
\[ \Hom(A^S, A) \cong \Hom(\textbf{2}^{ks}, \textbf{2}^k). \]
In general, for algebras $B$, $B_i$, $i \in I$, the set $\Hom(B, \prod_{i \in I} B_i)$ is in bijection with $\prod_{i \in I} \Hom(B,B_i)$ via $\phi \mapsto (\pi_i \circ \phi)_{i \in I}$, for any $\phi \in \Hom(B, \prod_{i \in I} B_i)$ (c.f. \cite[p. 70]{M98}). Hence, the result follows by Lemma \ref{le:boolean}:
\[ \vert \Hom(\textbf{2}^{ks}, \textbf{2}^k) \vert = \left\vert \prod_{i=1}^k \Hom(\textbf{2}^{ks}, \textbf{2}) \right\vert = \prod_{i=1}^k \left\vert \Hom(\textbf{2}^{ks}, \textbf{2}) \right\vert = \prod_{i=1}^k ks  = (ks)^k. \] 
\end{proof}  \bigskip

\textbf{Acknowledgments:} The first author of this paper was supported by a CONACYT Basic Science Grant (No. A1-S-8013) from the Government of Mexico. We thank the insightful comments made by the anonymous referee of this paper.


\end{document}